\documentclass[10pt, reqno]{amsart}
\numberwithin{equation}{section}
\usepackage{xcolor, cite}
\usepackage{amssymb}
\usepackage{hyperref}
\usepackage{graphicx}

\usepackage{mathabx}

\newcommand{\qtq}[1]{\quad\text{#1}\quad}

\theoremstyle{definition}
\newtheorem{definition}{Definition}

\theoremstyle{plain}
\newtheorem{theorem}[definition]{Theorem}
\newtheorem*{theorem*}{Theorem}
\newtheorem{lemma}{Lemma}

\newtheorem*{claim*}{Claim}
\newtheorem{corollary}[definition]{Corollary}

\theoremstyle{remark}
\newtheorem*{remark*}{Remark}

\newcommand{\eps}{\varepsilon}

\DeclareMathOperator{\R}{\mathbb{R}}

\DeclareMathOperator{\C}{\mathbb{C}}
\DeclareMathOperator{\F}{\mathcal{F}}

\newcommand{\jbrak}[1]{\langle#1\rangle}

\begin{document}

\title[Dispersion-managed NLS]{Well-posedness and blowup for the dispersion-managed nonlinear Schr\"odinger equation}

\author[J. Murphy]{Jason Murphy}
\address{Department of Mathematics \& Statistics, Missouri S\&T}
\email{jason.murphy@mst.edu}
\author[T. Van Hoose]{Tim Van Hoose}
\address{Department of Mathematics \& Statistics, Missouri S\&T}
\email{trvkdb@mst.edu}

\begin{abstract}
We consider the nonlinear Schr\"odinger equation with periodic dispersion management.  We first establish global-in-time Strichartz estimates for the underlying linear equation with suitable dispersion maps.  As an application, we establish a small-data scattering result for the $3d$ cubic equation.  Finally, we use a virial argument to demonstrate the existence of blowup solutions for the $3d$ cubic equation with piecewise constant dispersion map. 
\end{abstract}

\maketitle

%%%%%%%%%%%%%%%%%%%%%
%%%%%%%%%%%%%%%%%%%%%
\section{Introduction}

We study the initial-value problem for a certain class of \emph{dispersion-managed} nonlinear Schr\"odinger equations (DMNLS).  In general, these equations take the form
\begin{equation}\label{DMNLS}
i\partial_t u + \gamma(t)\Delta u + |u|^p u =0,\quad (t,x)\in\R\times\R^d,
\end{equation}
where the dispersion map $\gamma$ is some time-periodic real-valued function.  Such equations arise naturally in the context of nonlinear fiber optics, where one often encounters the cubic case $(p=2)$ in one dimension $(d=1)$ with piecewise constant dispersion map $\gamma$ as in \eqref{def:gamma} below (see e.g. \cite{Agr, Kur}).  For some representative mathematical results in this case (as well as some other cases), we refer the reader to \cite{CL, CHL, EHL, GT1, GT2, HL, HL2, MV, PZ, ZGJT}.  We note that in many of these works, one does not study \eqref{DMNLS} directly, but instead averages over one period of the dispersion map and studies the resulting autonomous equation.

The authors of \cite{AntonelliSautSparber} initiated the study of the initial-value problem for \eqref{DMNLS} for general dimensions and powers $p>0$ with piecewise constant $1$-periodic dispersion map of the form
\begin{equation}\label{def:gamma}
\gamma(t) = \begin{cases} \gamma_+ & 0\leq t< t_+ \\ \gamma_- & t_+\leq t<1,\end{cases}
\end{equation}
with $\gamma_\pm>0$.  Their results included a local well-posedness theory in $H^1$ for energy-subcritical equations (i.e. $p<\tfrac{4}{d-2}$); a global well-posedness result in $L^2$ for the mass-subcritical case (i.e. $p<\tfrac{4}{d}$); and a sharp global well-posedness result in $H^1$ for the mass-critical case (i.e. $p=\tfrac{4}{d}$), including the existence of finite-time blowup solutions at the sharp threshold.  The local theory appearing in \cite{AntonelliSautSparber} relies on the use of local-in-time Strichartz estimates for the usual linear Schr\"odinger equation.  Combining this with mass conservation yields the global result in the mass-subcritical case.  In the mass-critical case, the authors rely on the sharp Gagliardo--Nirenberg inequality to establish the global well-posedness result, while their blowup result relies on the use of the pseudoconformal symmetry. 

Our goal in this work is to initiate the study of the global behavior of solutions in the intercritical setting (that is, for powers between the mass- and energy-critical exponents), where so far only the local behavior has been understood.  We focus on the model case of the $3d$ cubic equation, namely
\begin{equation}\label{nls}
\begin{cases}
i\partial_tu + \gamma(t)\Delta u + |u|^2 u = 0, \\ u|_{t=0}=u_0\in H^1(\R^3),
\end{cases}
\end{equation}
with $\gamma$ as in \eqref{def:gamma}, although most of what we do carries over to the general intercritical case (i.e. $\tfrac{4}{d}<p<\tfrac{4}{d-2}$) in a straightforward way. 

Our first result is a global-in-time Strichartz estimate for dispersion-managed Schr\"odinger equations (see Theorem~\ref{T:Strichartz}), which may be of more general interest than the specific applications given here.  The proof applies to a class of dispersion maps that is somewhat more general than the piecewise constant case (see Definition~\ref{def:admissible}).  The most essential restrictions are (i) non-vanishing average dispersion, i.e. $\langle \gamma\rangle :=\int_0^1\gamma(t)\,dt\neq 0$, and (ii) non-vanishing of the dispersion map itself (e.g. $\gamma^{-1}\in L_t^\infty$).  A related NLS model in which the dispersion map itself vanishes at some points was considered in \cite{Fanelli}; it is possible that in this case, one could recover some Lorentz-space modified Strichartz estimates, although we did not pursue that direction here.  

As an application of the global Strichartz estimates, we establish a global well-posedness and scattering result for solutions to \eqref{nls} with non-zero average dispersion and \emph{small} initial data in $H^1$ (see Theorem~\ref{T:WP}).  On the other hand, we can also demonstrate the existence of solutions to \eqref{nls} that blow up in finite time (see Theorem~\ref{T:blowup}).  In particular, we may obtain finite-time blowup regardless of the sign of the average dispersion.  To prove the blowup result, we combine the virial identity with a scaling argument to demonstrate the existence of blowup solutions in any `focusing' step.  We briefly summarize our main results as follows (with the statements given in the main body of the paper: 
\begin{theorem*}\text{ }
\begin{itemize}
\item[1.] The standard Strichartz estimates hold for the linear equation
\[
i\partial_t u + \gamma(t)\Delta u = 0
\]
for admissible dispersion maps $\gamma$. See Theorem~\ref{T:Strichartz} and Theorem~\ref{T:inhomogeneous}.
\item[2.] For admissible dispersion maps $\gamma$, small initial data in $H^1$ lead to global solutions to \eqref{nls} that scatter.  See Theorem~\ref{T:WP}.
\item[3.] For piecewise constant dispersion maps $\gamma$, equation \eqref{nls} admits solutions that blow up in finite time.  See Theorem~\ref{T:blowup}.
\end{itemize}
\end{theorem*}

This combination of results, namely small-data scattering together with the existence of finite-time blowup solutions, leads to the interesting problem of identifying sharp conditions on the initial data that guarantee global well-posedness and scattering.  In the setting of \cite{AntonelliSautSparber}, the authors considered the mass-critical dispersion-managed NLS and were able to describe a sharp condition for global-wellposedness purely in terms of the conserved mass.  In the case of the standard intercritical NLS (e.g. \eqref{nls} with $\gamma\equiv 1$), the sharp condition for scattering versus blowup is described in terms of a combination of the mass and energy (see e.g. \cite{HR}).  In the dispersion-managed setting, one does not have a conserved energy; instead, a different energy is conserved on each interval on which $\gamma$ is constant.  This makes it challenging to adapt any type of `energy trapping' argument in order to propagate bounds even over one full period of the dispersion map. We plan to revisit this problem in a future work. 

The rest of the paper is organized as follows: In Section~\ref{S:Strichartz} we establish the global-in-time Strichartz estimates for the underlying linear model.  In Section~\ref{S:GWP}, we establish global well-posedness and scattering for sufficiently small data.  Finally, in Section~\ref{S:blowup}, we prove the blowup result. 

\subsection{Notation} We write $A \lesssim B$ or $B \gtrsim A$ to denote the inequality $A \leq CB$ for some $C > 0$.  We use $'$ to denote H\"older duals.  That is, for $p\in[1,\infty]$, we write $p'\in[1,\infty]$ for the solution to $\tfrac{1}{p}+\tfrac{1}{p'}=1$.  Our notation for the Fourier transform is
\[
  \F[f](\xi) = \hat{f}(\xi) := (2\pi)^{-\frac{d}{2}} \int_{\R^d} e^{i x \cdot \xi} f(x) \,dx.
\]
The Schr\"odinger group is then given by 
\[
 e^{i\cdot\Delta}=\F^{-1}e^{-i\cdot\xi^2}\F.
\]

%%%%%%%%%%%%%%%%%%%%%
%%%%%%%%%%%%%%%%%%%%%
\section{Strichartz estimates for the dispersion-managed equation}\label{S:Strichartz}

In this section we establish global-in-time Strichartz estimates for dispersion-managed Schr\"odinger equations of the form
\begin{equation}\label{DMLS}
\begin{cases} i\partial_t u + \gamma(t)\Delta u = 0, \\ u|_{t=s}=\varphi.\end{cases}
\end{equation}
The equation \eqref{DMLS} has the solution 
\begin{equation}\label{wiG}
u(t,s) = e^{i\Gamma(t,s)\Delta}\varphi,\qtq{where} \Gamma(t,s)=\int_s^t \gamma(\tau)\,d\tau. 
\end{equation}

Our main result addresses the following class of dispersion maps.  A typical example is shown in Figure~\ref{FIG}.
\begin{definition}[Admissible]\label{def:admissible} We call $\gamma:\R\to\R$ \emph{admissible} if it satisfies the following conditions:
\begin{itemize}
\item $\gamma$ is one-periodic: $\gamma(t+1)=\gamma(t)$ for all $t\in\R$.
\item $\gamma$ and $\gamma^{-1}$ are bounded:  $\|\gamma\|_{L^\infty}+\|\gamma^{-1}\|_{L^\infty}<\infty$. 
\item $\gamma$ has at most finitely many discontinuities on $[0,1]$.
\item $\gamma$ has nonzero average over its period: $\langle \gamma\rangle:=\int_0^1\gamma(t)\,dt\neq 0$.  
\end{itemize}
\end{definition}
\begin{figure}[h!]
\includegraphics[width=.5\linewidth]{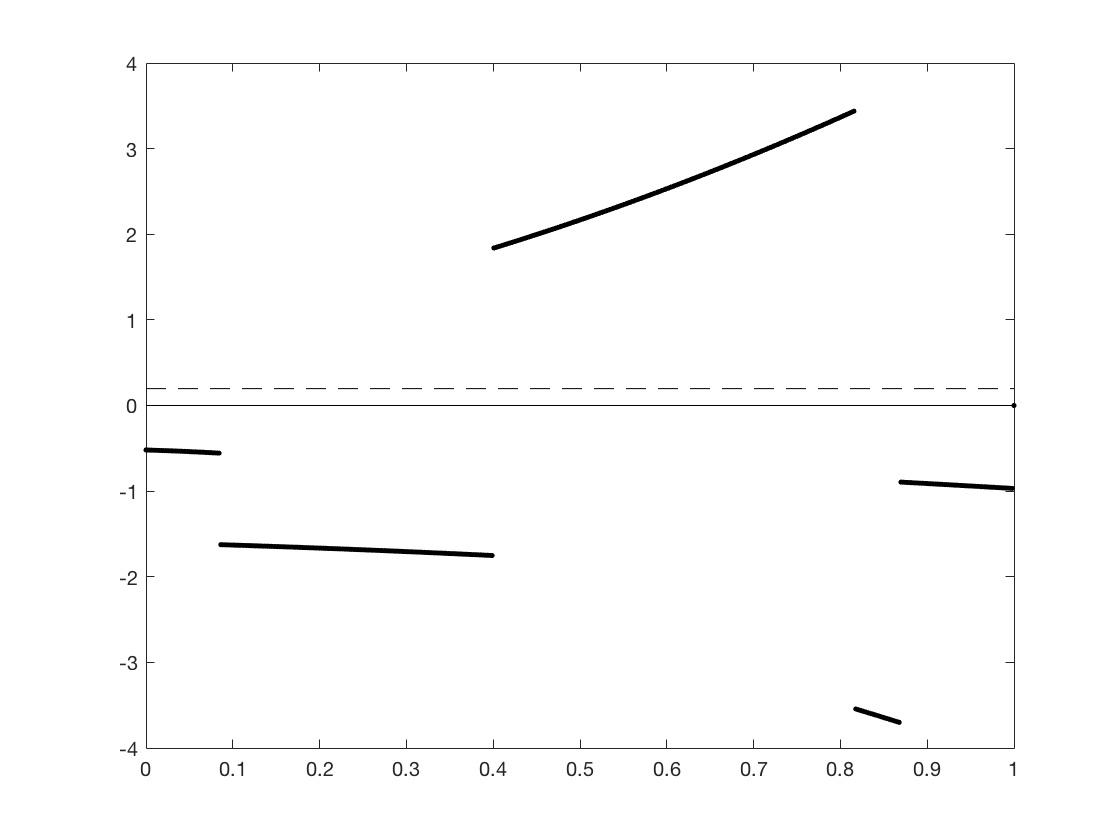}\includegraphics[width=.5\linewidth]{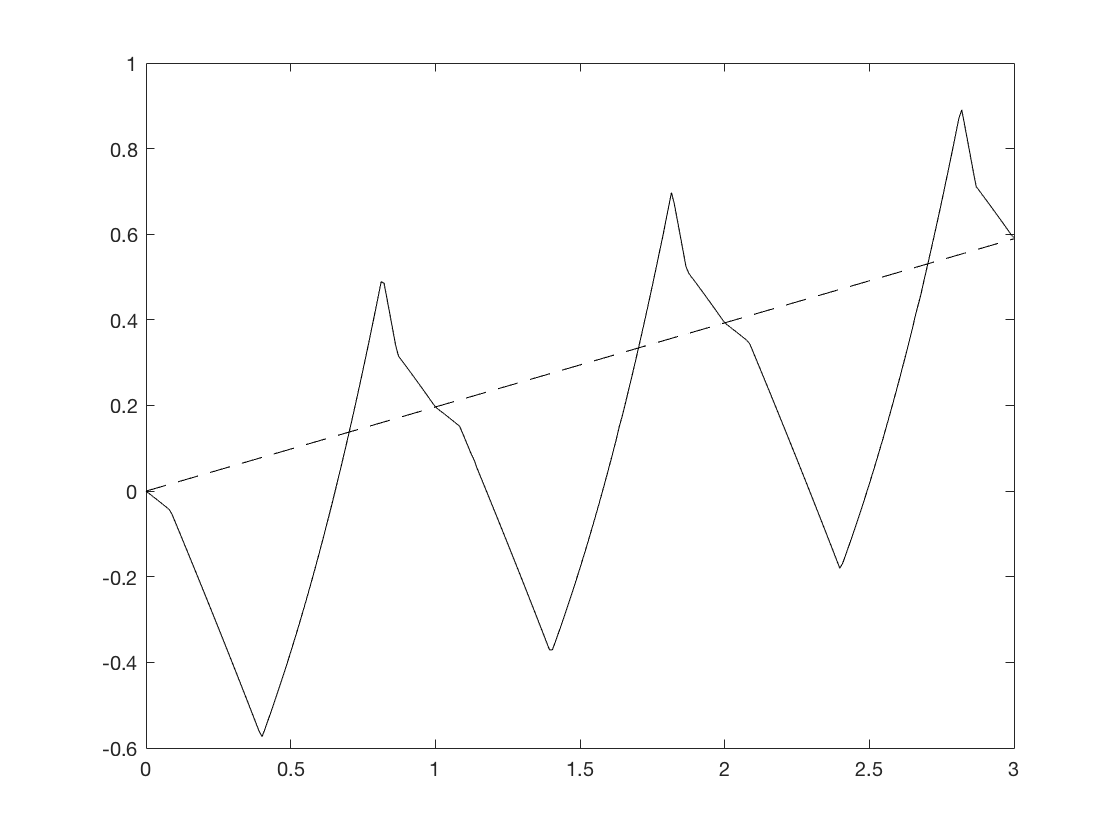}

\caption{On the left is plotted one period of $\gamma$. The solid line is the line $y=0$; the dotted line is $y=\langle\gamma\rangle$, which in this case is small and positive.  On the right is plotted $t\mapsto \Gamma(t,0)$ for $t\in[0,3]$; the dotted line corresponds to the line $y=\langle \gamma \rangle t$.}\label{FIG}
\end{figure}

The class of admissible functions includes the important piecewise constant case \eqref{def:gamma} in the case of nonzero average dispersion.  Indeed, Definition~\ref{def:admissible} is basically a slight generalization of this special case.  On the other hand, Definition~\ref{def:admissible} does \emph{not} permit the case that $\gamma\to 0$ along some sequence of times.  In particular, the proof of the global Strichartz estimates given below requires that $\gamma$ stay bounded away from zero.  A model of NLS in which $\gamma$ vanishes was considered in \cite{Fanelli}.  It is possible that some Lorentz-modified Strichartz estimates could be established in this setting (e.g. by modifying the estimate in \eqref{lorentz-maybe}), although we did not pursue that direction here.  Instead, our main result shows that for admissible functions $\gamma$, we can derive global Strichartz estimates for \eqref{DMLS} directly from those known to hold for the usual Schr\"odinger equation (cf. \cite{GinibreVelo, KeelTao, Strichartz}). 

\begin{theorem}[Strichartz estimates]\label{T:Strichartz} For any dimension $d\geq 1$, any $2\leq q,r\leq\infty$ satisfying
\begin{equation}\label{schrodinger-admissible}
\tfrac{2}{q}+\tfrac{d}{r}=\tfrac{d}{2}\qtq{and} (d,q,r)\neq(2,2,\infty),
\end{equation}
and any admissible $\gamma$ (in the sense of Definition~\ref{def:admissible}), there exists $C=C(\gamma)>0$ so that for any $\varphi\in L^2$ and $s\in\R$, we have 
\[
\|e^{i\Gamma(t,s)\Delta}\varphi\|_{L_t^q L_x^r(\R\times\R^d)}\leq C\|\varphi\|_{L^2(\R^d)},\qtq{where} \Gamma(t,s)=\int_s^t \gamma(\tau)\,d\tau.
\]
\end{theorem} 

\begin{remark*} The proof below will show that the constant may be taken to be
\begin{equation}\label{Cgamma}
C(\gamma)=C_{\text{Str}}\cdot\| \gamma^{-1}\|_{L^\infty}^{\frac{1}{q}}\cdot \biggl[1+N_\gamma\biggl\{1+\frac{[1+4\|\gamma\|_{L^\infty}]}{ |\langle \gamma\rangle|}\biggr\} \biggr]^{\frac{1}{q}},
\end{equation}
where $\langle \gamma\rangle = \int_0^1 \gamma(s)\,ds$ is the average dispersion; $N_\gamma$ is the number of discontinuities of $\gamma$ in $[0,1]$; and $C_{\text{Str}}$ is the constant for the standard $L_x^2(\R^d)\to L_t^q L_x^r(\R^d)$ Strichartz estimate.  In particular, our estimate breaks down when $\gamma$ becomes unbounded or approaches zero; when the average dispersion tends to zero; or when the number of discontinuities in one period becomes unbounded.
\end{remark*}

The proof relies on a few lemmas.  We begin with the following:

\begin{lemma}\label{L:lem1} Let $\gamma$ be a one-periodic function on $\R$. Define 
\[
\Gamma(t)=\int_0^t \gamma(s)\,ds \qtq{and}\langle\gamma\rangle=\int_0^1 \gamma(s)\,ds.
\]
Then 
\begin{equation}
\left|\Gamma(t) - t \jbrak{\gamma} \right| \leq 2\|\gamma\|_{L^\infty}\qtq{for all}t\in\R.
\end{equation}
\end{lemma}

\begin{proof} Writing $\lfloor t\rfloor$ for the floor of $t$ and using $1$-periodicity of $\gamma$, we have that
\begin{align*}
\Gamma(t)&= \lfloor t\rfloor \jbrak{\gamma} + \int_{\lfloor t \rfloor}^t \gamma(s) \,ds\\
& = t \jbrak{\gamma} + (t - \lfloor t\rfloor)\jbrak{\gamma} + \int_{\lfloor t \rfloor}^t \gamma(s)\,ds.
\end{align*}
As $|t-\lfloor t\rfloor|\leq 1$ for all $t\in\R$, this yields
\[
\bigl| \Gamma(t)-t\langle\gamma\rangle\bigr| \leq |\langle \gamma\rangle|+ \|\gamma\|_{L^\infty},
\]
which implies the result. \end{proof}

In particular, Lemma~\ref{L:lem1} implies (via the intermediate value theorem) that for $\langle \gamma\rangle\neq 0$, the range of $\Gamma(\cdot,0)$ equals $\R$.  We also have the following corollary, which plays a role when we later partition the range of $\Gamma(\cdot,0)$.

\begin{corollary}\label{L:lem2}  Let $\gamma$ be a one-periodic function on $\R$, with
\[
\Gamma(t)=\int_0^t \gamma(s)\,ds,\qtq{and}\langle \gamma\rangle=\int_0^1 \gamma(s)\,ds\neq 0.
\]
Then for any $\delta>0$ and any $t_1,t_2\in\R$, we have that 
\[
|t_2-t_1|> \frac{\delta+4\|\gamma\|_{L^\infty}}{|\langle \gamma\rangle|}\implies |\Gamma(t_2)-\Gamma(t_1)|>\delta. 
\]
\end{corollary}
\begin{proof} Without loss of generality, suppose that $\langle\gamma\rangle>0$ and $t_2>t_1$.  Then by Lemma~\ref{L:lem1},     
\[
\Gamma(t_2) - \Gamma(t_1) \geq (t_2 - t_1) \jbrak{\gamma} - 4\|\gamma\|_{L^\infty},
\]
which implies the result.\end{proof}

We turn to the proof of Theorem~\ref{T:Strichartz}.
\begin{proof}[Proof of Theorem~\ref{T:Strichartz}] We fix the dimension $d\geq 1$ and $(q,r)$ obeying \eqref{schrodinger-admissible}.  We let $\gamma$ be admissible in the sense of Definition~\ref{def:admissible} and suppose $\langle \gamma\rangle =\int_0^1\gamma(s)\,ds>0$. It will suffice to prove that
\[
\|e^{i\Gamma(t)\Delta}\varphi\|_{L_t^q L_x^r(\R\times\R^d)}\lesssim \|\varphi\|_{L^2},
\]
where $\varphi\in L^2$ and
\[
\Gamma(t)=\int_0^t \gamma(s)\,ds. 
\]

For convenience, we split $\R$ (viewed as the domain of $\Gamma$) into a disjoint union $\R=\mathbb{I}\cup\mathbb{D}$ so that $\Gamma$ is increasing on $\mathbb{I}$ and decreasing on $\mathbb{D}$ (equivalently, $\gamma$ is positive on $\mathbb{I}$ and negative on $\mathbb{D}$). It then suffices to establish $L_t^q L_x^r$ bounds on $\mathbb{I}\times\R^d$, say.  Note that $\mathbb{I}$ and $\mathbb{D}$ are themselves disjoint unions of intervals; moreover, if we write $N_\gamma$ for the number of discontinuities of $\gamma$ in one period, then the number of such intervals comprising $\mathbb{I}$ (or $\mathbb{D}$) within any unit interval is at most $N_\gamma+1$.

For each $n\in\mathbb{Z}$, we may now decompose
\[
\Gamma^{-1}([n,n+1))\cap \mathbb{I}=\bigcup_{k=1}^{K_n} U_k^n
\]
for some disjoint collection of intervals $\{U_k^n\}_{k=1}^{K_n}$, where each $U_k^n$ is contained in a distinct interval of $\mathbb{I}$. In particular, $\Gamma$ is injective when restricted to each $U_k^n$. The key in what follows is to bound $K_n$ uniformly in $n$.  In particular, we claim that
\begin{equation}\label{covering-number}
K_n \leq K_\gamma:=1+N_\gamma\biggl\{1+\frac{[1+4\|\gamma\|_{L^\infty}]}{ \langle \gamma\rangle}\biggr\}\qtq{for all}n\in\mathbb{Z}.
\end{equation}
Indeed, by Corollary~\ref{L:lem2}, we have that
\[
|\Gamma^{-1}([n,n+1))| \leq \frac{1+4\|\gamma\|_{L^\infty}}{\langle \gamma\rangle}\qtq{for each}n\in\mathbb{Z},
\]
where $|\cdot|$ denotes Lebesgue measure.  We now observe that each $U_k^n$ corresponds to a distinct interval in $\mathbb{I}$, and hence to at least one discontinuity of $\gamma$ (at least, if $K_n>1$).  In particular, if \eqref{covering-number} were to fail, then we could find more than $N_\gamma (1+C)$ discontinuities of $\gamma$ in a subset of $\R$ of total length less than $C$, a contradiction.

Now, using the change of variables $s=\Gamma(t)$ on each $U_k^n$ and the standard $L_x^2\to L_t^q L_x^r$ Strichartz estimates, we estimate
%  
%
%    To finish the proof, note that for any Schr\"odinger admissible pair $(q, r)$ we have that 
%    \begin{align*}
%        &\left\| e^{i\Gamma(t) \Delta} u_0 \right\|_{q,r} = \\
%        &\qquad \int_{\R} \|e^{i \Gamma(t) \Delta} u_0 \|_{L_x^r}^q \,dt \\
%        &= \int_{\Gamma(\I)} \|e^{i \Gamma(t) \Delta} u_0 \|_{L_x^r}^q \,dt + \int_{\Gamma(\D)} \|e^{i \Gamma(t) \Delta} u_0 \|_{L_x^r}^q \,dt 
%    \end{align*}
%    We now present the proof of the result for only the first term; i.e. the one corresponding to the part where $\Gamma$ is increasing. Note that by the additivity property of the Lebesgue integral, we know that 
\begin{align}
\int_{\mathbb{I}} \|e^{i \Gamma(t) \Delta} \varphi \|_{L_x^r}^q \,dt  & \leq \sum_{n\in\mathbb{Z}}\sum_{k=1}^{K_n} \int_{U_k^n} \|e^{i \Gamma(t) \Delta} \varphi \|_{L_x^r}^q \,dt\nonumber \\
& \leq \sum_{n\in\mathbb{Z}} K_\gamma\|\gamma^{-1}\|_{L^\infty} \int_{n}^{n+1} \|e^{is\Delta}\varphi\|_{L_x^r}^q\,ds\label{lorentz-maybe} \\
& \leq  K_\gamma\|\gamma^{-1}\|_{L^\infty}\|e^{it\Delta}\varphi\|_{L_t^q L_x^r(\R\times\R^d)}^q \nonumber \\
& \leq C_{\text{Str}}^qK_\gamma\|\gamma^{-1}\|_{L^\infty}\|\varphi\|_{L^2}^q,\nonumber
\end{align}
where $C_{\text{Str}}$ is the constant in the usual $L^2(\R^d)\to L_t^q L_x^r(\R^d)$ Strichartz estimate. 
\end{proof}

With Theorem~\ref{T:Strichartz} in place, we may obtain the full range of Strichartz estimates (other than the double $L_t^2$ endpoint, which we did not pursue here) via the method of $TT^*$ and the Christ--Kiselev lemma (see \cite{CK}).  In particular, we obtain the following:

\begin{theorem}[Inhomogeneous Strichartz estimates]\label{T:inhomogeneous} Let $d\geq 1$ and let $q,\tilde q,r,\tilde r$ obey \eqref{schrodinger-admissible} with $(q,\tilde q)\neq(2,2)$. Let $\gamma$ be admissible in the sense of Definition~\ref{def:admissible}.  Then for any $t_0\in\R$ we have the estimates
\begin{align*}
\biggl\| \int_{\R} e^{i\Gamma(t_0,t)\Delta}F(s)\,ds\biggr\|_{L^2(\R^d)} & \leq C\|F\|_{L_t^{q'} L_x^{r'}(\R\times\R^d)}, \\
\biggl\| \int_{t_0}^t e^{i\Gamma(t,s)\Delta}F(s)\,ds\biggr\|_{L_t^q L_x^r(\R\times\R^d)}&\leq C \|F\|_{L_t^{\tilde q'}L_x^{\tilde r'}(\R\times\R^d)}
\end{align*}
for suitable $C=C(\gamma)$, where $\Gamma(t,s)=\int_s^t \gamma(\tau)\,d\tau.$
\end{theorem}

%%%%%%%%%%%%%%%%%%%%%
%%%%%%%%%%%%%%%%%%%%%
\section{Well-posedness}\label{S:GWP}

In this section, we consider the initial-value problem
\begin{equation}\label{E:DMNLS}
    \begin{cases}
        i\partial_t u + \gamma(t) \Delta u + |u|^2 u=0 \\
        u|_{t=t_0} = u_0 \in H^1(\R^3)
    \end{cases}
\end{equation}
on $\R\times\R^3$, where $\gamma$ is an admissible function in the sense of Definition~\ref{def:admissible}.  This includes the important special case  
\begin{equation}\label{E:dispersion}
    \gamma(t) = \begin{cases}
        \gamma_+, & 0 < t < t_+ \\
        -\gamma_- & t_+ < t < 1
    \end{cases}
\end{equation}
for $\gamma_\pm>0$, extended periodically to $\R$, provided the average dispersion 
\begin{equation}\label{E:avgdisp}
    \jbrak{\gamma} := \int_0^1 \gamma(s)\,ds
\end{equation}
is nonzero. As in the previous section, we define
\begin{equation}\label{E:capgamma}
\Gamma(t, s) = \int_s^t \gamma(\tau) \,d\tau,
\end{equation}
so that the Duhamel formula for the solution to \eqref{E:DMNLS} is given by
\[
u(t)= e^{i\Gamma(t,t_0)\Delta}u_0 +i\int_{t_0}^t e^{i\Gamma(t,s)\Delta}(|u|^2u)(s)\,ds.
\]

The local $H^1$ theory (i.e. local existence for $H^1$ initial data) for \eqref{E:DMNLS}--\eqref{E:dispersion} was previously considered in \cite{AntonelliSautSparber}.  In this section, we apply the Strichartz estimates obtained in the previous section to establish a {global} result, namely, global well-posedness and scattering for sufficiently small initial data in $H^1$.  In fact, with the global Strichartz estimates in hand, the proof follows from a fairly standard contraction mapping argument. 

\begin{theorem}\label{T:WP} Let $\gamma$ be an admissible function in the sense of Definition~\ref{def:admissible} and define
\[
\Gamma(t,s)=\int_s^t \gamma(\tau)\,d\tau.
\]
Then there exists $\eta_0=\eta_0(\gamma)>0$ such that the following holds:  Given $t_0\in\R$ and $u_0\in H^1(\R^3)$, if
\[
\bigl\| |\nabla|^{\frac12} e^{i\Gamma(t,t_0)\Delta} u_0\|_{L_t^5 L_x^{\frac{30}{11}}([t_0,\infty)\times\R^3)}<\eta<\eta_0,
\]
then there exists a unique forward-global solution $u:[t_0,\infty)\times\R^3\to\C$ to \eqref{E:DMNLS}, which scatters in the sense that 
\[
\lim_{t\to\pm\infty} \|u(t)-e^{i\Gamma(t,t_0)\Delta}u_+\|_{H^1} = 0
\] 
for some $u_+\in H^1$.  Analogous statements hold backward in time. 
%
%
%
%
% Let $u(t_0, x) := u_0 \in H^1_x(\R^d)$. Then there exists $\eps^* > 0$ so that if $0 < \eps \leq \eps^*$ and if the estimate
%    \begin{equation}
%        \left\||\nabla|^{\frac{1}{2}} e^{i\Gamma(t, t_0)\Delta} u_0 \right\|_{L_t^5 L_x^{30/11}(\R \times \R^3)} \lesssim \eps 
%    \end{equation}
%    holds, then there exists a unique solution $u$ to the equation \eqref{E:DMNLS} that obeys the following spacetime bounds:
%    \begin{align*}
%        \left\||\nabla|^{1/2} u\right\|_{L_t^5 L_x^{30/11}} &\lesssim 2C\eps\\
%        \left\| u \right\|_{L_t^\infty H_x^1} &\lesssim 2C \|u_0\|_{H_x^1}\\
%        \left\| \jbrak{\nabla} u \right\|_{L_t^5 L_x^{30/11}} &\lesssim 2C \|u_0\|_{H_x^1}
%    \end{align*}
%    and scatters in $H^1$, in the sense that there exists a unique $u_+ \in H^1_x$ such that 
%    \begin{align*}
%        \|u(t) - e^{i\Gamma(t, t_0)\Delta}u_+ \|_{H_x^1} \to 0
%    \end{align*}
%    as $t \to \infty$. 
\end{theorem}
\begin{proof}  The proof is based on a contraction mapping argument using the Strichartz estimates established in the previous section.  We define 
    \begin{equation}\label{E:Duhamel}
        \Phi(u)(t) := e^{i \Gamma(t, t_0)\Delta}u_0 + i \int_{t_0}^t e^{i\Gamma(t, s)\Delta} |u|^2 u(s) \,ds.
    \end{equation}
We let $u_0\in H^1$ and set $A=\|u_0\|_{H^1}$. Given $\eta>0$ to be determined below, we define the complete metric space
\[
X=\{u:\|u\|_{L_t^\infty H_x^1}\leq 2CA,\ \|u\|_{L_t^5 H_x^{1,\frac{30}{11}}}\leq 2CA,\ \||\nabla|^{\frac12} u\|_{L_t^5 L_x^{\frac{30}{11}}}\leq 2\eta\},
\]
with
\[
d(u,v)=\|u-v\|_{L_t^5 L_x^{\frac{30}{11}}}\qtq{for}u,v\in X.
\]
Here and below space-time norms are taken over $[t_0,\infty)\times\R^3$, and $C>0$ is chosen to encode the implicit constants arising in Sobolev embedding and Strichartz estimates; in particular, the dependence of $\eta$ on $\gamma$ arises through the implicit constants in the Strichartz estimates (see e.g. \eqref{Cgamma} above).
%    \begin{equation}
%        \left\{
%            u \ \left| \ 
%        \begin{aligned}
%            &\| |\nabla|^{1/2} u \|_{L_t^5 L_x^{30/11}} \lesssim 2C\eps\\
%            &\| u \|_{L_t^\infty H_x^1} \lesssim 2C \|u_0\|_{H_x^1}\\
%            &\| \jbrak{\nabla} u \|_{L_t^5 L_x^{30/11}} \lesssim 2C\| u_0 \|_{H_x^1}
%        \end{aligned}
%        \right.
%        \right\}
%    \end{equation}
 
%It is clear by a standard dominated-convergence type argument that $(X, d)$ is a closed subspace of $L_t^\infty H_x^1 \cap L_t^5 W_x^{1, 30/11}$, which is a complete space, so that $(X, d)$ is complete. We now need to show that $\Phi : X \to X$ is a contraction for $\eps$ small enough, and that the data-to-solution map given by $\Phi$ is $C_t H_x^1$. To see that $\Phi$ is an endomorphism of $X$, note that we have the estimates

To show that $\Phi:X\to X$, we first let $u\in X$ and use Strichartz and Sobolev embedding to estimate 
    \begin{align*}
        \| \Phi(u)\|_{L_t^\infty H_x^1} &\lesssim \|u_0\|_{H_x^1} + \left\| \int_{t_0}^t e^{i\Gamma(t, s)\Delta} |u|^2 u(s) \,ds \right\|_{L_t^\infty H_x^1}\\
        & \lesssim A + \bigl\| \langle\nabla\rangle\bigl[|u|^2u\bigr]\|_{L_t^\frac{5}{3}L_t^{\frac{30}{23}}} \\
        &\lesssim A + \|u\|_{L_{t,x}^{5}}^2 \|\jbrak{\nabla}u\|_{L_t^5 L_x^{\frac{30}{11}}}\\
        &\lesssim A + 8C\eta^2 A\leq 2CA
    \end{align*}
for suitable $C>0$ and $\eta=\eta(C)$ sufficiently small.   We then obtain the same estimate for the $L_t^5 H_x^{1,\frac{30}{11}}$-norm of $\Phi(u)$, as well.  

For the $L_t^5 \dot H_x^{\frac12,\frac{30}{11}}$ estimate, we begin with an application of Sobolev embedding.  Using the fractional chain rule, we obtain 
    \begin{align*}
        \| |\nabla|^{\frac12} \Phi(u) \|_{L_t^5 L_x^{\frac{30}{11}}} &\leq \| |\nabla|^{\frac12} e^{i \Gamma(t, t_0) \Delta}u_0 \|_{L_t^5 L_x^{\frac{30}{11}}} + \biggl\| |\nabla|^{\frac12} \int_{t_0}^t e^{i\Gamma(t, s) \Delta} |u|^2 u(s)\,ds \biggr\|_{L_t^5 L_x^{\frac{30}{11}}} \\
        &\leq \eta + C\|u\|_{L_{t,x}^5}^2 \||\nabla|^{\frac12} u\|_{L_t^5 L_x^{\frac{30}{11}}} \\
        &\leq \eta + 8C^4\eta^3 \leq 2\eta
    \end{align*}
provided $\eta$ is sufficiently small. 

To show that $\Phi$ is a contraction, we let $u,v\in X$ and estimate as above to deduce
    \begin{align}
        \|\Phi(u)- \Phi(v)\|_{L_{t,x}^5} &\lesssim \| |u|^2 u - |v|^2 v \|_{L_t^{\frac{5}{3}}L_x^{\frac{30}{23}}} \\
        &\lesssim \{\|u\|_{L_{t,x}^5}^2+\|v\|_{L_{t,x}^5}^2\} \|u - v\|_{L_t^5 L_x^{\frac{30}{11}}} \\
        &\lesssim \eta^2 \|u - v\|_{L_t^5 L_x^{ \frac{30}{11}}},
    \end{align}
which implies that $\Phi$ is a contraction for $\eta$ sufficiently small. 

It follows that $\Phi$ has a unique fixed point $u$, yielding the desired solution.  To prove the scattering result, it suffices to show that $\{e^{i\Gamma(t_0,t)\Delta}u(t)\}$ is Cauchy in $H^1$ as $t\to\infty$.  To this end, we observe that by the Duhamel formula \eqref{E:Duhamel},
\[
e^{i\Gamma(t_0,t)\Delta}u(t)-e^{i\Gamma(t_0,s)\Delta}u(s) = i\int_s^t e^{i\Gamma(t_0,\tau)\Delta}|u|^2 u(\tau)\,d\tau. 
\]
Thus, applying the Strichartz estimates and estimating as above, we obtain
\begin{align*}
\|e^{i\Gamma(t_0,t)\Delta}u(t)-e^{i\Gamma(t_0,s)\Delta}u(s)\|_{H^1} & \lesssim \| \langle \nabla \rangle [|u|^2 u]\|_{L_t^{\frac{5}{3}}L_x^{\frac{30}{23}}((s,t)\times\R^3)} \\
& \lesssim \|u\|_{L_{t,x}^5((s,t)\times\R^d)}^2 \| u\|_{L_t^5 H_x^{1,\frac{30}{11}}((s,t)\times\R^d)} \\
& \to 0 \qtq{as}s,t\to\infty,
\end{align*}
which yields the result.  \end{proof}

%%%%%%%%%%%%%%%%%%%%%
%%%%%%%%%%%%%%%%%%%%%
\section{Finite-time Blowup}\label{S:blowup}

In this section, we continue to consider the equation \eqref{E:DMNLS} but restrict attention only to the piecewise constant case \eqref{E:dispersion}.  We will adapt the virial argument to demonstrate the possibility of finite-time blowup solutions.  In particular, we will exhibit solutions that blow up on the first interval $[0,t_+)$, although the same argument would apply on any `focusing' step.  The existence of local-in-time solutions follows from \cite{AntonelliSautSparber} (or from suitable modifications of Theorem~\ref{T:WP} above), and so we will take the existence of solutions for granted and focus on the issue of blowup.  

To state our main result, we firstly introduce the \emph{ground state} solution $Q$ for the standard cubic NLS, which is the unique positive, radial, decreasing solution to 
\[
-Q+\Delta Q + Q^3 = 0
\]
and plays a key role in the determination of sharp scattering/blowup results in that setting (see e.g. \cite{HR}). Fixing $\gamma_+>0$, we then define
\[
R_+(x) = Q(\tfrac{x}{\sqrt{\gamma_+}}),
\]
which solves
\begin{equation}\label{Relliptic}
-R_++\gamma_+\Delta R_++R_+^3 =0.
\end{equation}
In particular, $u(t,x)=e^{it}R_+(x)$ solves \eqref{E:DMNLS}--\eqref{E:dispersion} on $[0,t_+)$. 

We next define the \emph{mass}
\[
M(u) = \int_{\R^3} |u|^2\,dx 
\]
and the \emph{energies}
\[
E_\pm(u) = \int_{\R^3}  \frac{\gamma_\pm}{2} |\nabla u|^2 \mp \frac14 |u|^4\,dx. 
\]
We observe that solutions to \eqref{E:DMNLS}--\eqref{E:dispersion} conserve the mass, while neither of $E_\pm(u)$ is globally conserved.  Instead, we have conservation of $E_+$ on intervals $[n,n+t_+)$ and conservation of $E_-$ on intervals $[n+t_+,n+1)$, where $n\in\mathbb{Z}$.

Our result is the following:

\begin{theorem}\label{T:blowup}  Suppose that $u_0\in H^1(\R^3)$ satisfies
\begin{equation}\label{MER}
\begin{aligned}
M(u_0)E_+(u_0)&<M(R_+)E_+(R_+), \\
\|u_0\|_{L^2}\|\nabla u_0\|_{L^2}& \geq \|R_+\|_{L^2}\|\nabla R_+\|_{L^2}.
\end{aligned}
\end{equation}
Suppose further that either $xu_0\in L^2$ or $u_0$ is radial.  Then there exists $\lambda>0$ sufficiently large that the solution to \eqref{E:DMNLS}--\eqref{E:dispersion} with $u|_{t=0}=\lambda u_0(\lambda\cdot)$ blows up at some time $T\in(0,t_+)$. 
\end{theorem}

The condition \eqref{MER} would be a typical blowup condition for the cubic NLS
\begin{equation}\label{gamma+NLS}
i\partial_t u + \gamma_+ \Delta u + |u|^2 u = 0,
\end{equation}
with the proof following the standard virial argument (see e.g. \cite{Glassey, HR}).  In the present setting, the basic idea is to combine this argument with scaling to make the blowup happen before the end of the first focusing step (that is, before time $t_+$).

The key identity is the following standard virial identity:
\begin{equation}\label{E:virial}
\begin{aligned}
\tfrac{d^2}{dt^2}\int_{\R^3} |x|^2 |u(t,x)|^2\,dx & = \tfrac{d}{dt} \biggl[4\Im \int_{\R^3} \bar u \nabla u\cdot x\,dx\biggr] \\ 
\\ & = 8\int_{\R^3} \gamma_+ |\nabla u|^2 - \tfrac34 |u|^4\,dx
\end{aligned}
\end{equation}
for solutions to \eqref{gamma+NLS} (see e.g. \cite{HR} or \cite[Section~6.5]{Caz}).  The role of the condition \eqref{MER} is to guarantee that the right-hand side of \eqref{E:virial} is quantitatively negative throughout the lifespan of the solution:

\begin{lemma}\label{L:virial} Suppose $u_0\in H^1$ satisfies
\begin{align}
M(u_0)E_+(u_0)&\leq (1-\delta) M(R_+)E_+(R_+), \label{trp0} \\
\|u_0\|_{L^2}\|\nabla u_0\|_{L^2}& \geq \|R_+\|_{L^2}\|\nabla R_+\|_{L^2}\label{trp}
\end{align}
for some $\delta>0$. Let $u$ be the solution to \eqref{gamma+NLS} with $u|_{t=0}=u_0$.  Then there exists $\delta'=\delta'(\delta)>0$, $c=c(\delta)>0$, and $\eps=\eps(\delta)>0$ so that
\begin{equation}\label{trp2}
\|u(t)\|_{L^2}\|\nabla u(t)\|_{L^2}  \geq (1+\delta') \|R_+\|_{L^2}\|\nabla R_+\|_{L^2},
\end{equation}
and
\begin{equation}\label{trp3}
\int_{\R^3}\gamma_+(1+\eps)|\nabla u(t,x)|^2 - \tfrac{3}{4}|u(t,x)|^4\,dx < -c
\end{equation}
uniformly for $t$ in the lifespan of $u$.\end{lemma}

\begin{proof} This is the standard `energy trapping' argument (see e.g. \cite[Theorem~4.2]{HR}).  The key observation is that the ground state $Q$ (and hence the rescaled ground state $R_+(\cdot)=Q(\tfrac{\cdot}{\sqrt{\gamma_+}}))$ is an optimizer for the sharp Gagliardo--Nirenberg inequality
\[
\|u\|_{L^4(\R^3)}^4 \leq C_0 \|u\|_{L^2(\R^3)}\|\nabla u\|_{L^2(\R^3)}^3,
\]
so that in particular
\begin{equation}\label{defC0}
C_0=\frac{\|R_+\|_{L^4}^4}{\|R_+\|_{L^2}\|\nabla R_+\|_{L^2}^3}.
\end{equation}
We can connect the various norms of $R_+$ to one another via the following Pohozaev identities (obtained by multiplying \eqref{Relliptic} by $R_+$ and $x\cdot\nabla R_+$ and integrating):
\begin{align*}
-\|R_+\|_{L^2}^2-\gamma_+\|\nabla R_+\|_{L^2}^2 +\|R_+\|_{L^4}^4 =0,& \\
\tfrac32\|R_+\|_{L^2}^2 + \tfrac{\gamma_+}{2}\|\nabla R_+\|_{L^2}^2 - \tfrac34\|R_+\|_{L^4}^4=0.&
\end{align*}
In particular,
\begin{equation}\label{E+R+}
\tfrac{\gamma_+}{3}\|\nabla R_+\|_{L^2}^2 = \tfrac{1}{4}\|R_+\|_{L^4}^4,\qtq{so that}E_+(R_+)=\tfrac{\gamma_+}{6}\|\nabla R_+\|_{L^2}^2. 
\end{equation}

We can therefore use the Gagliardo--Nirenberg inequality to obtain
\begin{align*}
(1-\delta)M(R_+)E_+(R_+) & \geq M(u)E_+(u) \\
& \geq \tfrac{\gamma_+}2 \|u(t)\|_{L^2}^2\|\nabla u(t)\|_{L^2}^2 -\tfrac{C_0}{4}\|u(t)\|_{L^2}^3\|\nabla u(t)\|_{L^2}^3,
\end{align*}
which, using \eqref{defC0} and \eqref{E+R+}, implies
\[
(1-\delta)\geq 3\biggl[\frac{\|u(t)\|_{L^2}\|\nabla u(t)\|_{L^2}}{\|R_+\|_{L^2}\|\nabla R_+\|_{L^2}}\biggr]^2 - 2\biggl[\frac{\|u(t)\|_{L^2}\|\nabla u(t)\|_{L^2}}{\|R_+\|_{L^2}\|\nabla R_+\|_{L^2}}\biggr]^3
\]
Thus, by \eqref{trp} and a continuity argument, we deduce that \eqref{trp2} holds.  

For \eqref{trp3}, we use \eqref{E+R+}, \eqref{trp0}, and \eqref{trp2} to write
\begin{align*}
\int \gamma_+ (1+\eps)|\nabla u|^2 - \tfrac{3}{4}|u|^4 \,dx & = \tfrac{1}{M(u)}\bigl[3M(u)E_+(u)-\gamma_+(\tfrac12-\eps)\|u(t)\|_{L^2}^2\|\nabla u(t)\|_{L^2}^2\bigr] \\
& \leq \frac{\gamma_+\|R_+\|_{L^2}^2\|\nabla R_+\|_{L^2}^2}{2M(u)}\bigl[(1-\delta)-(1+\delta')^2(1-2\eps)\bigr] \\
& = -\frac{\gamma_+\|R_+\|_{L^2}^2\|\nabla R_+\|_{L^2}^2}{2M(u)}\bigl[\delta+2\delta'+(\delta')^2-\eps(1+\delta')^2],
\end{align*}
which, choosing $\eps$ sufficiently small, yields the result. \end{proof}

We turn to the proof of Theorem~\ref{T:blowup}. 

\begin{proof}[Proof of Theorem~\ref{T:blowup}]  We take $u_0\in H^1$ such that \eqref{MER} holds.  

We first consider the case $xu_0\in L^2$. We let $u$ denote the maximal-lifespan solution to \eqref{E:DMNLS}--\eqref{E:dispersion} with initial data $u_0$, and for $\lambda\geq 1$ we let $u^\lambda$ denote the maximal-lifespan solution with initial data $u_0^\lambda:=\lambda u_0(\lambda\cdot)$.  We let $I_\lambda$ denote the intersection of the lifespan of $u^\lambda$ with $[0,\lambda^{-2}t_+]$.  In particular,
\[
u^\lambda(t,x)=\lambda u(\lambda^2 t,\lambda x)
\]
for all $t\in I_\lambda$.  Thus, by Lemma~\ref{L:virial} and scaling, there exists $c>0$ so that 
\begin{equation}\label{virialbd}
\int_{\R^3} \gamma_+ |\nabla u^\lambda(t,x)|^2 - \tfrac{3}{4}|u^\lambda(t,x)|^4\,dx < -\lambda c 
\end{equation}
uniformly for $t\in I_{\lambda}$ (here the $\eps$ improvement in Lemma~\ref{L:virial} is not needed). 

We now set
\[
f_\lambda(t)=\int_{\R^3} |x|^2 |u^\lambda(t,x)|^2\,dx\qtq{for}t\in I_\lambda
\]
and write
\begin{equation}\label{expandft}
f_\lambda(t)=f_\lambda(0)+tf_\lambda'(0)+ \int_0^t \int_0^s f_{\lambda}''(\tau)\,d\tau\,ds.
\end{equation}
We observe (cf. \eqref{E:virial} and \eqref{virialbd}) that
\begin{equation}\label{ready-to-blow}
\begin{aligned}
&|f_\lambda(0)| = C_1\lambda^{-3}, \\
&|f_\lambda'(0)|  = \biggl|4\gamma_+\Im\int_{\R^3} \bar u_0^\lambda \nabla u_0^\lambda \cdot x\,dx\biggr| \leq C_2\lambda^{-1}, \\
&f_\lambda''(\tau)  < -c\lambda,
\end{aligned}
\end{equation}
uniformly for $\tau\in I_\lambda$, where $C_1:=\|xu_0\|_{L^2}^2$ and $C_2:=4\|xu_0\|_{L^2}\|\nabla u_0\|_{L^2}$.  We therefore obtain 
\[
0\leq f_\lambda(t) \leq -\tfrac12c\lambda t^2 + C_2\lambda^{-1}t+C_1\lambda^{-3}\qtq{for}t\in I_\lambda.
\]
We now observe that the quadratic polynomial on the right-hand side equals zero at time
\[
T_\lambda:=\frac{C_2+\sqrt{C_2^2+2cC_1}}{c\lambda^2}.
\]
Choosing $\lambda$ large enough that $T_{\lambda}<t_+$, we therefore obtain that the solution $u^\lambda$ blows up at or before time $T_\lambda$.

We next consider the case that $u_0$ (and hence the solution $u$) is radial.  In this case, we use a localized version of the virial identity.  As before, we consider the rescaled solution $u^\lambda$ on $I_\lambda$.  We then introduce a weight $w_R(x)=R^2\phi(\tfrac{x}{R})$, where $\phi$ is a smooth, nonnegative, radial function satisfying
\[
\phi(x)=\begin{cases} |x|^2 & |x|\leq 1 \\ \text{constant} & |x|\geq 3\end{cases}
\]
and obeying the bounds
\[
|\nabla\phi(x)|\leq 2|x| \qtq{and}|\partial_{jk}\phi(x)|\leq 2 \qtq{for all}x\in\R^3. 
\]
We will later specialize to $R\sim \lambda^{-1}$.

Proceeding as above, we define
\[
f_\lambda(t) = \int w_R(x)|u^\lambda(t,x)|^2\,dx>0
\]
and use the expansion \eqref{expandft}.  Using the fact that
\[
f_\lambda'(t) = 4\gamma_+\Im\int_{\R^3} \bar u^\lambda \nabla u^\lambda\cdot\nabla(w_R)\,dx,
\]
we first derive the bounds
\begin{equation}\label{radialf}
|f_\lambda(0)| \lesssim R^2\lambda^{-1} \qtq{and} |f_\lambda'(0)| \lesssim R. 
\end{equation}

Computing the second derivative of $f_\lambda$, we derive the following analogue of \eqref{E:virial}:
\begin{align}
f_\lambda''(t) & = 8\int \gamma_+ |\nabla u^\lambda|^2 - \tfrac34 |u^\lambda|^4 \,dx \label{virial-1} \\
&\quad + 4\gamma_+\Re\int_{|x|>R}  (\partial_j \bar u^\lambda)( \partial_k u^\lambda)\partial_{jk}[w_R]\,dx - 8 \int_{|x|>R}|\nabla u^\lambda|^2 \,dx \label{virial-2}\\
&\quad +\mathcal{O}\biggl[ \int_{|x|>R}R^{-2}|u^\lambda|^2 + |u^\lambda|^4 \,dx \biggr] \label{virial-3} 
\end{align}
(see e.g. \cite[Section~4]{HR}). 

Now, \eqref{virial-2}$\leq 0$ by the Cauchy--Schwarz inequality and the assumptions on $w_R$, while for the first term in \eqref{virial-3}, we use conservation of mass to obtain
\[
R^{-2}\|u^\lambda\|_{L^2}^2 \lesssim R^{-2}\lambda^{-1} \qtq{uniformly on}I_\lambda. 
\] 
For the remaining term, we use the radial Sobolev embedding estimate, Young's inequality, and the conservation of mass to obtain
\begin{align*}
\|u^\lambda\|_{L^4(|x|>R)}^4 & \leq R^{-2} \|u^\lambda\|_{L^2}^2 \| |x|u^\lambda \|_{L^\infty}^2 \\
& \leq CR^{-2} \|u^\lambda\|_{L^2}^3 \|\nabla u^\lambda\|_{L^2} \\
& \leq 8\eps \gamma_+ \|\nabla u^\lambda\|_{L^2}^2  + \tfrac{C}{\eps\gamma_+} R^{-4} \|u^\lambda\|_{L^2}^6 \\
& \leq 8\eps\gamma_+\|\nabla u^\lambda\|_{L^2}^2 + \tfrac{C}{\eps\gamma_+} R^{-4}\lambda^{-3},
\end{align*}
where we have allowed the constant $C$ to change in each line and $\eps$ is as in Lemma~\ref{L:virial}. Thus, continuing from above and using Lemma~\ref{L:virial}, we obtain
\begin{align*}
f_\lambda''(t) &\leq 8\int (1+\eps)\gamma_+|\nabla u^\lambda|^2 - \tfrac34 |u^\lambda|^4\,dx + CR^{-2}\lambda^{-1}+\tfrac{C}{\eps\gamma_+}R^{-4}\lambda^{-3} \\
& \leq -2c\lambda + CR^{-2}\lambda^{-1}+\tfrac{C}{\eps\gamma_+}R^{-4}\lambda^{-3}
\end{align*}
for some $c>0$. In particular, choosing
\[
R=C_0\lambda^{-1} \qtq{for sufficiently large} C_0=C_0(c,\eps,\gamma_+)
\]
and recalling \eqref{radialf}, we derive
\[
|f_\lambda(0)|\leq C_1\lambda^{-3},\quad |f_\lambda'(0)|\leq C_2\lambda^{-1},\qtq{and} f_\lambda''(t) \leq -c\lambda
\]
for some $c,C_1,C_2>0$, uniformly on $I_\lambda$.  As this puts us in exactly the same situation as \eqref{ready-to-blow} above, we deduce that for sufficiently large $\lambda$, we obtain blowup before time $t_+$. \end{proof}

\subsection*{Acknowledgements} We are grateful to Rowan Killip for helpful suggestions regarding the global Strichartz estimate. J.M. was supported by a Simons Collaboration Grant.

\end{document}